\documentclass[
final
, nomarks
]{dmtcs-episciences}

\usepackage[utf8]{inputenc}
\usepackage{subfigure}

\usepackage[round]{natbib}

\usepackage{graphicx}

\usepackage{enumerate}
\usepackage{mathbbol}
\usepackage{amssymb}
\usepackage{braket}
\usepackage{longtable}
\usepackage{caption}
\usepackage{mathtools}
\usepackage{microtype}
\usepackage{tikz}

\DeclareSymbolFontAlphabet{\mathbb}{AMSb}
\DeclareSymbolFontAlphabet{\mathbbl}{bbold}

\def\@cproof[#1]{\noindent\textbf{Proof of Claim #1:} \ignorespaces}
\def\@@cproof{\noindent\textbf{Proof of Claim:} \ignorespaces}
\AtBeginDocument{
	\def\cproof{\medskip\@ifnextchar[\@cproof\@@cproof}
	
}

\newtheorem{lemma}{Lemma}[section]
\newtheorem{theorem}[lemma]{Theorem}
\newtheorem{corollary}[lemma]{Corollary}

\newtheorem{question}{Question}

\newtheorem{assumption*}{Assumption}

\newtheorem{definition}[lemma]{Definition}

\newtheorem{claim}{Claim}
\newtheorem{claim1}{Claim}

\numberwithin{equation}{section}

\newcommand{\N}{\mathbb{N}}

\newcommand{\bs}{\backslash}

\renewcommand\AA{{\mathcal A}}

\newcommand\CC{{\mathcal C}}

\newcommand\EE{{\mathcal E}}

\newcommand\KK{{\mathcal K}}

\newcommand\PP{{\mathcal P}}
\newcommand\QQ{{\mathcal Q}}

\newcommand\TT{{\mathcal T}}

\newcommand{\tile}{\mathfrak t}

\def\Ind#1#2{#1\setbox0=\hbox{$#1x$}\kern\wd0\hbox to 0pt{\hss$#1\mid$\hss}
	\lower.9\ht0\hbox to 0pt{\hss$#1\smile$\hss}\kern\wd0}

\def\notind#1#2{#1\setbox0=\hbox{$#1x$}\kern\wd0
	\hbox to 0pt{\mathchardef\nn=12854\hss$#1\nn$\kern1.4\wd0\hss}
	\hbox to 0pt{\hss$#1\mid$\hss}\lower.9\ht0 \hbox to 0pt{\hss$#1\smile$\hss}\kern\wd0}

\usepackage{graphics}
\usepackage[outdir=./]{epstopdf}

\author{Samuel Braunfeld\affiliationmark{1}}
\title{The undecidability of joint embedding for 3-dimensional permutation classes}

\affiliation{ University of Maryland, College Park, USA}

\received{2020-02-28}
\revised{2021-03-25, 2021-07-23}
\accepted{2021-08-10}
\begin{document}
	\publicationdetails{22}{2021}{2}{10}{6165}
	\maketitle

\begin{abstract}
As a step towards resolving a question of Ru\v{s}kuc on the decidability of joint embedding for hereditary classes of permutations, which may be viewed as structures in a language of 2 linear orders, we show the corresponding problem is undecidable for hereditary classes of structures in a language of 3 linear orders.
\end{abstract}


\section{Introduction}
In \cite{Rusk}, Ru\v{s}kuc posed several decision problems for finitely-constrained permutation classes, with the decidability of atomicity among them (and this question was recently re-posed in \cite{Jel}). A permutation avoidance class is called \textit{atomic}
  if it cannot be expressed as a union of two proper subclasses. A general hope is that understanding a permutation class can be reduced to understanding its atomic subclasses, as in the following lemma for calculating growth rates (see \cite{Vatter} for a reference). 
  
  \begin{lemma}
   Suppose $\KK$ is a permutation class, with no infinite antichain in the containment order. Then $\KK$ can be expressed as a finite union of atomic subclasses. Furthermore, the upper growth rate of $\KK$ is equal to the maximum upper growth rate among its atomic subclasses. 
   \end{lemma}
   
   We may view permutations as structures in a language of two linear orders. Atomicity is then equivalent to the joint embedding property (see \cite{Vatter}), a standard model-theoretic notion, so we may rephrase Ru\v{s}kuc's question.
   
   \begin{definition}
   A class $\CC$ of structures has the \textit{joint embedding property (JEP)} if, given $A, B \in \CC$, there exists $C \in \CC$ such that $A,B$ embed into $C$.
   \end{definition}
   
    \begin{question} \label{q:Ruskuc}
    Is there an algorithm that, given finite set of forbidden permutations, decides whether the corresponding permutation class has the joint embedding property?
    \end{question}
    
    This problem is known to be decidable in certain restricted classes of permutations, such as monotone grid classes \cite{Waton}.
      Also, whether a permutation class is a natural class, which is a strengthening of atomicity, is decidable \cite{Murphy}.
      
      However, we believe there is a strong possibility Ru\v{s}kuc's problem is undecidable in general. We are not aware of many undecidability results in the permutation class literature, although \cite{Pak}, using methods that seem quite different from ours, proves an undecidability result about comparing the parity of the number of permutations of size $n$ in two permutation classes. 
      
      The author took a first step towards Ru\v{s}kuc's problem in \cite{Braunf}, proving the JEP is undecidable for hereditary graph classes. Although it is not yet clear whether that proof can be adapted to permutations, we here adapt it to {\em 3-dimensional permutations}, i.e. structures in a language of 3 linear orders, proving the following theorem via a reduction from the string tiling problem.
        
        \begin{theorem}
        There is no algorithm that, given a finite set of forbidden 3-dimensional permutations, decides whether the corresponding 3-dimensional permutation class has the JEP.
        \end{theorem}   
        
        A very rough sketch of the proof is as follows. We use a reduction from the (string) tiling problem, which asks whether we can tile an infinite grid using a given collection of tile types, subject to local constraints. The first two steps below ensure that the tiling problem is equivalent to whether we can jointly embed two particular 3-dimensional permutations, and the third step ensures that joint embedding for the class is equivalent to joint embedding for those two 3-dimensional permutations.
                
                \begin{enumerate}
                \item Construct two 3-dimensional permutations $A^*$, representing a grid, and $B^*$ representing a suitable collection of tiles.
                \item Choose a finite set of constraints to ensure that if $C$ jointly embeds $A^*$ and $B^*$, then it encodes a solution to the string tiling problem.
                \item Show that if the string tiling problem admits a solution, then the chosen class has joint embedding.
                \end{enumerate}
    
\section{Background}    
\subsection{The (string) tiling problem}

Rather than using a reduction from the halting problem to prove undecidability, we will use the string tiling problem, a variant of the tiling problem. The input to a tiling problem consists of a finite set $Tiles$ of tile types, as well as a set of rules of the form ``Tiles of type $i$ cannot be placed directly above tiles of type $j$'' and ``Tiles of type $k$ cannot be placed directly right of tiles of type $\ell$''. A solution to a tiling problem is a surjective function $\tau \colon \N^2 \to Tiles$, interpreted as placing tiles on a grid, that respects the tiling rules. 

\begin{theorem}[\cite{Berger}] \label{thm:Berger}
There is no algorithm that, given a sets of tile types and tiling rules, decides whether the corresponding tiling problem has a solution.
\end{theorem}

We will use a variant, called string tiling problems in \cite{WQO}. Here there are only two tile types, but there is some $D \in \N$ such that for every $d \leq D$, tiling rules may restrict which tiles are placed at distance $d$ to the right of a given tile, or directly above a given tile. An encoding of tiling problems as string tiling problems is given in Lemma 7.6 of \cite{WQO}, the idea being to use several tiles in the string tiling problem to encode a single tile from the standard tiling problem. This proves the analogue of Theorem \ref{thm:Berger} for the string tiling problem.

As we will be reducing from the string tiling problem, which is co-recursively enumerable, we point out here that if $\CC$ is a hereditary class of finite structures in a finite relational language, then the JEP for $\CC$ is also co-recursively enumerable. To see this, consider $A, B \in \CC$ that can be jointly embedded, as witnessed by $C \in \CC$ and embeddings $f \colon A \to C$ and $g \colon B \to C$. As $\CC$ is hereditary, the substructure of $C$ induced on $f(A) \cup g(B)$ is also in $\CC$. Thus, given $A, B \in \CC$, $|A|+|B|$ is a bound on the size of the possible witnesses for joint embedding, and they can be exhaustively checked.

\subsection{The argument for hereditary graph classes} \label{sub:graphs}

We will now sketch the argument from \cite{Braunf} for hereditary graph classes (in an expanded language with colored vertices and edges, and both directed and undirected edges), since our argument in this paper will attempt to re-encode it using 3-dimensional permutations. Although we are concerned with the JEP for finite structures in a hereditary class $\CC$, the compactness theorem implies that the JEP for the finite members of $\CC$ is equivalent to the JEP for countable members of $\CC$. Rather than work with families of increasingly large finite structures, we prefer to take our canonical models to be countable. 

\begin{figure}[h]
\begin{center}
\includegraphics[scale=.6]{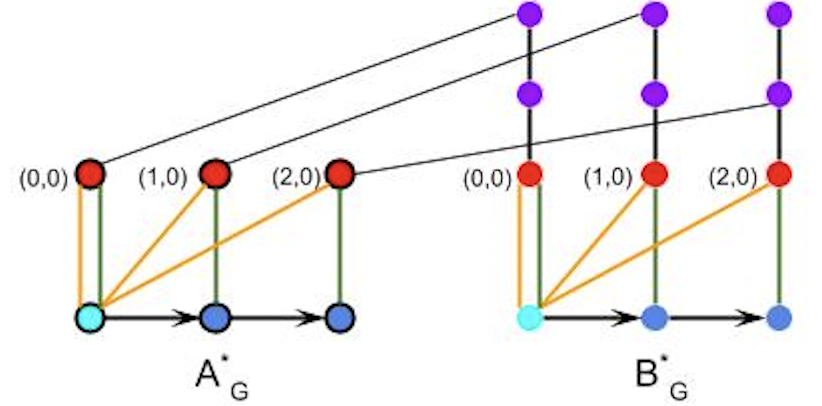}
\end{center}
\caption{A portion of the canonical models $A^*_G$ and $B^*_G$, with the grid points in $A^*_G$ tiled by tiles attached to grid points with the same coordinates in $B^*_G$. Path points are blue, with the origin a different shade. Grid points are red, their $y$-coordinate determined by an orange edge and their $x$-coordinate by a green edge. Tile points are purple. Points in 0-superscripted predicates have a black border, while points in 1-superscripted predicates do not. \\ This encodes a tiling of (0,0) with tile-type 2, (1,0) with tile-type 2, and (2,0) with tile-type 1.}
\label{fig:graph canonical}
\end{figure}

Suppose we are given a tiling problem $\TT$. First, we describe graphs corresponding to $A^*$ and $B^*$ from the rough sketch in the introduction. $A^*_G$ (see Figure \ref{fig:graph canonical}) will contain a 1-way infinite directed path. To every pair of points in this path, we attach a point, representing a grid point with coordinates taken from the attached path points. Because we must distinguish between $x$ and $y$-coordinates, we use the colored edges to attach each grid point to its coordinates. Furthermore the path points are colored distinctly from the grid points, and the origin of the path is also colored distinctly. $B^*_G$ will look like a copy of $A^*_G$, although using a disjoint set of vertex colors. Furthermore, to each grid point in $B^*_G$, path of length $\tile$ (where $\tile$ is the number of tile types in the given tiling problem), using a new color for these points. These represent a full tile set  available at each coordinate, with the different tile-types being distinguished by their distance from the corresponding grid point.

We then choose our constraints so that  when we try to jointly embed $A^*_G$ and $B^*_G$, the following is forced: for every grid point in $A^*_G$, with coordinates $(x, y)$, we must add an edge to one tile point attached to the grid point in $B$ with the same coordinates. This is interpreted as tiling the point $(x, y)$ by the corresponding tile-type, and our constraints should further enforce the local tiling rules.

As $A^*_G$ and $B^*_G$ will be in our hereditary class $\CC_\TT$, if $\CC_\TT$ has the JEP, then $\TT$ must have a solution, since we can read a valid tiling off the structure embedding $A^*_G$ and $B^*_G$. We must then show that if $\TT$ has a solution $\tau \colon \N^2 \to Tiles$, then we may jointly embed any $A, B \in \CC_\TT$. For this, we add a variety of additional constraints ensuring that if we must add edges due to the constraints in the previous paragraph, and thus are attempting to encode a valid tiling, then $A$ and $B$ look approximately like one of our canonical models $A^*_G$ and $B^*_G$. Crucially, we ensure that every grid point involved in our attempted tiling has unique coordinates $(x,y)$ on a unique path; we thus have a well-defined input to give to $\tau$, and add edges from grid points in $A$ to tile points in $B$ (or vice versa) as $\tau$ dictates.

The additional difficulties with (3-dimensional) permutations arise from the transitivity of the orders, which places severe limitations on how we may jointly embed a given pair of structures. Also, some concerns that are in common with the graph case shift in their difficulty. A key point in the graph case is that grid points and tile sets have unique coordinates. While that was simple to enforce in the graph case, it, and even the proper definition of coordinates, will be a significant concern here. On the other hand, the point of most concern in the graph case was ensuring that none of the configurations used to encode unary predicates were accidentally created by our joint embedding procedure, i.e. our method for constructing a third structure jointly embedding two given structures. Here this problem will be trivialized by taking advantage of the third linear order, but it returns to the fore when working with permutation classes.

\section{The canonical models} 

\subsection{Preliminary definitions}

We first mention that the primary reason for using a third linear order is to obtain the first claim at the beginning of Lemma \ref{lemma:tilingtoJEP}. The third order can largely be ignored otherwise, which may help in picturing the constructions.

\begin{figure}[h]
	\begin{center}
		\includegraphics[scale=.35]{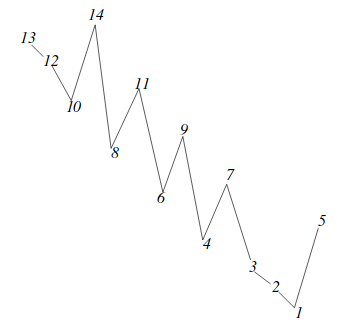}
	\end{center}
	\caption{A typical antichain element from \cite{SB}. The root is 13.}
	\label{fig:antichain}
\end{figure}

We choose an antichain $\AA$ of 3-dimensional permutations on which ${<_1} = {<_3}^{opp}$, i.e. the opposite order of $<_3$, and containing at least 20 members (10 are used for the unary predicates in this section, and we will double that number later), which we will use to encode unary predicates, which may be thought of as coloring points. We also require that each element of $\AA$ have at least 5 points, and that the $<_1$-greatest, $<_1$-least, $<_2$-greatest, and two $<_2$-least points of each element are distinct, with the $<_1$-greatest point $<_2$-below the $<_1$-least point. For example, let $\AA$ be the infinite antichain from \cite{SB} (see Figure \ref{fig:antichain}; in this figure and all others, $<_1$ is horizontal and $<_2$ is vertical), with the third order defined by ${<_1} = {<_3}^{opp}$.

For $i \in \set{0,1}$, select distinct antichain elements $E_X^i, E_Y^i, E_P^i, E_G^0, E_T^1$, and $E_O^i$, and let $\EE$ be the set of these members. Elements with a 0-superscript will be used to encode grid points, i.e. structures like the graph $A^*_G$ and their coordinates, while those with a 1-superscript will encode tile sets and their coordinates, i.e. structures like the graph $B^*_G$.

If $E \in \EE$, we say $x$ is the \textit{root} of $E$ if it is the $<_1$-least point. We will think of roots as actually representing points in the colored graph we are trying to encode, and almost all other points as being auxiliary to assist the encoding.

 We also define the following unary predicates.
\begin{enumerate}
\item (path points) $x \in P^i$ if $x$ is the root of a copy of $E_P^i$ or $E_O^i$
\item (path origins) $x \in O^i$ if $x$ is the root of a copy of $E_O^i$
\item (grid points) $x \in G^0$ if $x$ is the root of a copy of $E_G^0$
\item (tile-type 1) $x \in T_1^1$ if $x$ is the root of a copy of $E_T^1$
\item (tile-type 2) $x \in T_2^1$ if $x$ is the $<_2$-greatest point of a copy of $E_T^1$
\item (tile points) $T^1 = T_1^1 \cup T_2^1$
\end{enumerate}

Note that we do not encode points representing tile-type 2 by roots. Since these will always be paired with points representing tile-type 1, we instead represent tiles by two different points in $E_T^1$.

The antichain condition is meant to stop the following kind of situation. Suppose that $E_P^0$ embedded into $E^G_0$. Then whenever we encode a grid point by adding a copy of $E^G_0$, we would also be adding a copy of $E^P_0$, and thus unintentionally be adding a path point. 

In addition to encoding unary predicates, we will use elements of $\EE$ to encode edges between their roots and other points, using the following notion of {\em capture}.

Given a point $x$ and $E$ a copy of an element in $\EE \bs \set{E^1_T}$, we say $x$ is \textit{captured} by $E$ if $x$ is $<_2$-between the two $<_2$-least points of $E$, $E <_1 x$, and $E <_3 x$. This should be thought of as encoding a graph edge between $x$ and the root of $E$.

An example of capture is shown in Figure \ref{fig:capture}. There are two copies of an element of $\EE$ projected onto $<_1$ and $<_2$, although the left copy should be taken $<_3$-below the right copy.  We view this as encoding a directed edge from the root of the left copy to the root of the right copy, and we may encode a directed path by continuing to daisy-chain such copies.

\begin{figure}[h]
	\begin{center}
		\includegraphics[scale=.1]{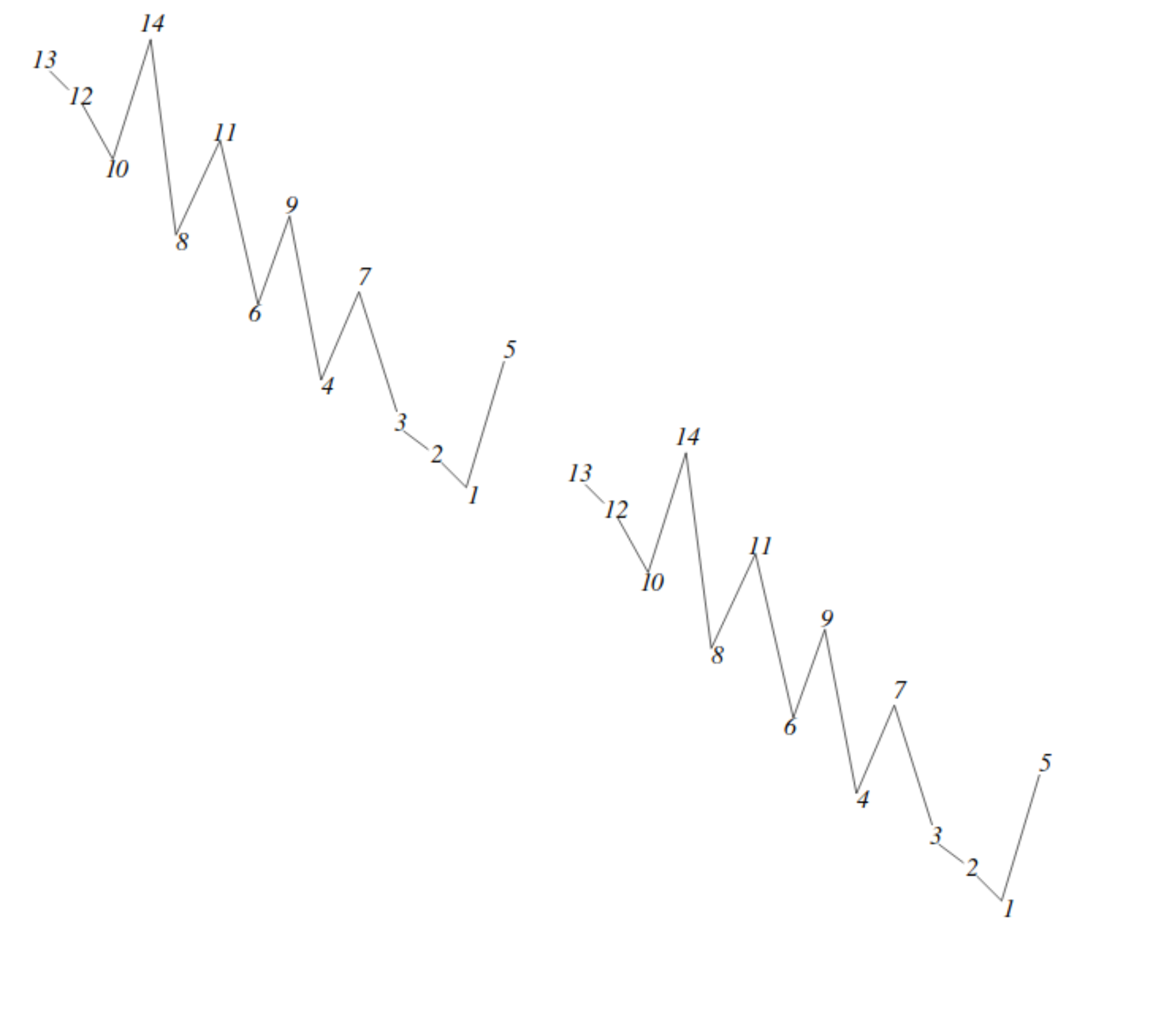}
	\end{center}
	\caption{Two copies of an element from $\EE$. The left copy is capturing the root of the right copy.}
	\label{fig:capture}
\end{figure}

Before giving the remaining definitions, we discuss the stylized portion of the canonical models shown in Figure \ref{fig:canonical3}, which should be compared to the graph version in Figure \ref{fig:graph canonical}. We have projected onto $<_1$ and $<_2$, and we will later see that $<_3$ will be determined by $<_1$. Since the precise structure of the elements of $\EE$ is unimportant, copies of those elements are represented by colored lines with two $<_2$-least points marked to display capture. As in Figure \ref{fig:graph canonical}, the blue points represent path points, with the origin a different shade. Grid points are in red and tile points in purple, with their $x$-coordinates determined by the path point captured by the copy of $E_X^i$ denoted in green, and their $y$-coordinates determined by the path point captured by the copy of $E_Y^i$ denoted in orange. Pairs of tiles at a given coordinate are connected by a line, with the tile of type 2 above and right of the tile of type 1. Finally, the copies of $E_G^0$ are the black lines connected to grid points, and they capture a tile from the pair of tiles at the corresponding coordinate. If the dividing line between $A^*_{<_1}$ and $B^*_{<_1}$ is ignored, this picture encodes a tiling of both $(0,0)$ and $(1,0)$ with tile-type 2.  The dotted horizontal lines separate different regions of the picture, but are not part of the structure.

\begin{figure}[h]
	\begin{center}
	\includegraphics[scale=.4]{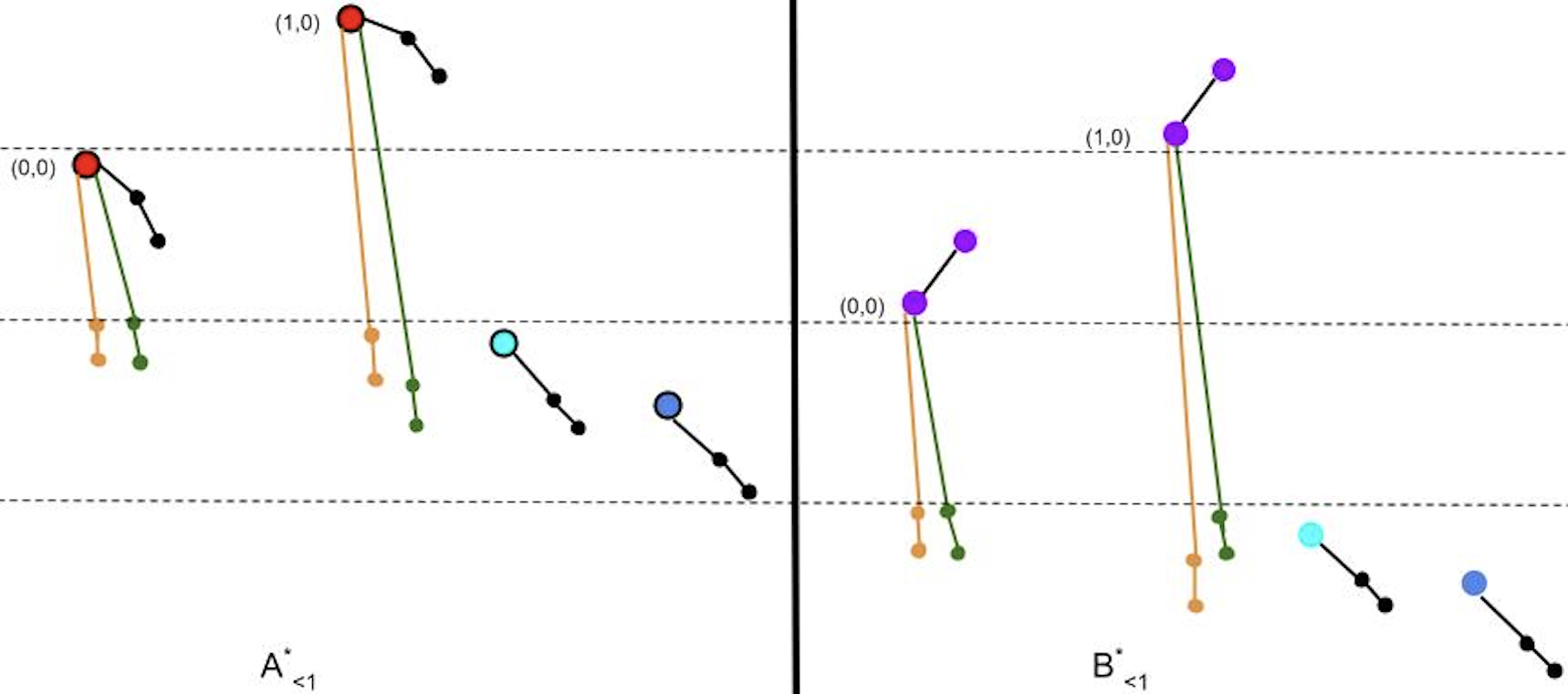}
	\end{center}
	\caption{A stylized projection onto $<_1$ and $<_2$ of a portion of the canonical models $A^*_{<_1}$ and $B^*_{<_1}$.}
	\label{fig:canonical3}
\end{figure}

The following definitions may all be seen in Figure \ref{fig:canonical3}, with the exception that the $P^i$-paths shown only have length 2 instead of being infinite.

We say  $x \in G^0$ {\em is tiled by} $y \in T^1$ if $x$ is the root of a copy of $E_G^0$ that captures $y$.

We say $t_1 \in T^1_1$ and $t_2 \in T^1_2$ \textit{form a tile set} if there exists $E$ a copy of $E^1_T$ with root $t_1$ and $<_2$-greatest point $t_2$. 

Given a point $g \in G^0$ and $x, y \in P^0$ or $g \in T_1^1$ and $x, y \in P^1$, we say $g$ is \textit{coordinatized} by $(x, y)$ if $g$ is the root of a copy of $E^0_X$ that captures $x$ and of $E^0_Y$ that captures $y$ (or in the second case, we use $E^1_X$ and $E^1_Y$).

We say $p'$ is a \textit{path-successor} of $p$ if $p, p' \in P^i$ and $p$ is the root of a copy of $E_P^i$ or $E_O^i$ that captures $p'$. 

We say $h$ is a \textit{horizontal successor} of $g$ if $g, h \in G^0$ or $g, h \in T_1^1$ and there are $x, y, x' \in P^0$ (or $P^1$ in the second case) such that $g$ is coordinatized by $(x, y)$, $h$ is coordinatized by $(x', y)$, and $x'$ is a path-successor of $x$. \textit{Vertical successor} is defined similarly, but $h$ is coordinatized by $(x, y')$ and $y'$ is a path-successor of $y$. {\em Horizontal predecessor} and {\em vertical predecessor} are defined conversely.

We define an \textit{infinite one-way $P^i$-path} to be a copy of $E_O^i$ with root $p_0$ and a sequence of copies of $E^i_P$, with roots $(p_1, p_2, \dots)$ arranged such that $p_{k+1}$ is captured by the copy of $E^i_P$ (or $E_O^i$) rooted at $p_k$, and the copy of $E_{p}^i$ (or $E_O^i$) rooted at $p_k$ is $<_1$-below that rooted at $p_{k+1}$. In this case, we say $p_0$ is the \textit{path-origin} of the path.

We say $g \in G^0$ is a \textit{grid-origin} if there is some $x \in O^0$ such that $G$ is coordinatized by $(x, x)$. We say $t \in T^1_1$ is a \textit{tile-origin} if there is some $x \in O^1$ such that $G$ is coordinatized by $(x, x)$.

Let $g \in G^0 \cup T_1^1$ be coordinatized by $(x, y)$. We say $g$ is \textit{on the x-axis} if $y \in O^i$ (for the appropriate $i$), and $g$ is \textit{on the y-axis} if $x \in O^i$ (we will sometimes also refer to a tile set being on an axis if its first tile is). Note that a grid-origin or tile-origin is on both the $x$-axis and $y$-axis.

We define a \textit{connector interval} to be the open $<_2$-interval defined by the two $<_2$-least points of a copy of $E^0_G$. We define a \textit{tile set interval} to be the open $<_2$-interval defined by the $<_1$-least point and the $<_2$-greatest point (i.e. by the two tiles) of a copy of $E^1_T$. Finally, we define a \textit{special interval} to be either a connector interval or a tile set interval. 

Given a special interval defined by some $E \in \EE$, we call the $<_2$-greater endpoint of the special interval its \textit{top endpoint}, and the $<_2$-lesser endpoint its \textit{bottom endpoint}.

Although we defined a special interval as a $<_2$-interval corresponding to a copy of an element of $\EE$, we will often conflate the special interval with its corresponding copy of an element of $\EE$. However, intersection of special intervals will always refer to intersection of the $<_2$-intervals.

\subsection{The canonical models} \label{sub:canon}
Our proof proceeds in two steps. First we prove the undecidability of the $<_1$-JEP, defined below, following the rough sketch from the introduction. This will be finished by Section \ref{sub:<1}.  Then, in Section \ref{sub:jep}, we quickly reduce from the $<_1$-JEP to the JEP. In this section, we describe our canonical models for the first step.

\begin{definition}
We say that a class of 3-dimensional permutations has the \textit{$<_1$-JEP} if it admits a joint embedding procedure in which, given factors labeled $A$ and $B$, the procedure places $A <_1 B$.
\end{definition}

 We now describe our canonical models $A^*_{<_1}$ and $B^*_{<_1}$ for the $<_1$-JEP, corresponding to the graphs $A^*_G$ and $B^*_G$ from \S \ref{sub:graphs}. We only describe $<_1$ and (sometimes) $<_2$, since $<_3$ will be determined as follows: if $x, y$ are in the same copy of an element of $\EE$ that we specify below, then $x <_1 y \iff x >_3 y$. Otherwise, $x <_1 y \iff x <_3 y$. Note that this will ensure that the only  copies of elements of $\EE$ appearing in either $A^*_{<_1}$ or $B^*_{<_1}$  will be those specified below. For suppose there is some further, unintended copy $\widehat E$ of an $E \in \EE$. By the antichain condition, it cannot embed into any single copy of an $E' \in \EE$, so there must be $x, y \in \widehat E$ occurring in two distinct specified copies of elements of $\EE$. Suppose that $x <_1 y$, so then $x <_3 y$. But this contradicts that $<_1$ and $<_3$ disagree within $E$.

It would be helpful to keep in mind the image of part of $A^*_{<_1}$ and $B^*_{<_1}$ given in Figure \ref{fig:canonical3} when reading the following construction. We start constructing $A^*_{<_1}$ by placing an infinite one-way $P^0$-path with roots $(p_0, p_1, \dots)$. Then, $<_1$-below and $<_2$-above the path, we place a sequence of points indexed by $\N^2$, increasing antilexicographically with respect to $<_1$ and $<_2$ (we say $(x, y) <_{antilex} (x', y')$ if $y < y'$ or $y=y'$ and $x < x'$); so, identifying a point with its indices, if $g <_{antilex} g'$ then $g <_{1,2} g'$. We now make each such point, which we will call grid points, the root of 3 different copies of elements of $\EE$. Consider the point $g$ indexed by $(x, y)$.
 We make $g$ the root of a copy $E_{g, X}$ of $E^0_X$, $E_{g, Y}$ of $E^0_Y$, and $E_{g, G}$ of $E^0_G$, satisfying the following.
\begin{enumerate}
\item $E_{g,X}$ captures $p_x$ and $E_{g,Y}$ captures $p_y$.
\item $E_{g, G}$ is $<_2$-above the path.
\item $E_{g,X} <_1 E_{g,Y} \bs \set{g} <_1 E_{g,G} \bs \set{g}$.
\item Let $g <_{antilex} g'$. Then for every $E, E' \in \EE$ (not necessarily distinct), any copy of $E$ rooted at $g$ is $<_1$-less than any copy of $E'$ rooted at $g'$. Furthermore, $E_{g, G} <_2 E_{g',G}$.
\end{enumerate}

We construct $B^*_{<_1}$ similarly, except using $1$-superscripted elements of $\EE$ instead of $0$-superscripted elements, and using copies of $E^1_T$ instead of $E^0_G$.

As in the graph case, we will choose our constraints so that when performing the $<_1$-JEP on $A^*_{<_1}$ and $B^*_{<_1}$, i.e. constructing a $C$ that jointly embeds both, we will be forced to tile each grid point in $A^*_{<_1}$ by a tile from the corresponding tile set in $B^*_{<_1}$.

\subsection{Picturing joint embedding} In this subsection, we refer back to Figure \ref{fig:canonical3} to help picture the process of performing $<_1$-joint embedding. This should explain some choices we made in constructing our canonical models and should help motivate some of the constraints in the next section. Also, although Figure \ref{fig:canonical3} depicts the canonical models, what we say will apply to any 3-dimensional permutations we consider.

Suppose we want to $<_1$-joint embed the structures $A=A^*_{<_1}$ and $B=B^*_{<_1}$ depicted in Figure \ref{fig:canonical3}. We must place $A <_1 B$. We will also choose to place $A<_3 B$ since, as in the construction of our canonical models, this will mean we don't create unintended copies of elements of $\EE$. Given this choice, our only freedom is in determining the $<_2$-relation between points in $A$ and points in $B$. We may view our procedure as erasing the dividing line between the two structures, keeping the points in $A$ fixed, and keeping the horizontal positions of the points in $B$ fixed but being able to slide them around vertically while maintaining their vertical ordering. By sliding the points around, we may make certain special intervals in $A$ and $B$ intersect, and in particular make certain tiles fall into corresponding connector intervals.

Figure \ref{fig:canonical3} was drawn so that minimal sliding is required. If we just erase the line and leave all the points fixed, this produces a tiling of both $(0,0)$ and $(1, 0)$ by 2-tiles. If we instead wish to tile $(0,0)$ by a 1-tile, we may slide the two tiles at $(0,0)$ upwards while keeping all other points fixed until the 1-tile falls into the connector interval at $(0,0)$ in $A$ and the 2-tile falls out. (Since we are keeping the tiles at $(1, 0)$ fixed, the tiles at $(0,0)$ will also be ``compressed'' closer together as we slide them upward. But this is an artifact of the picture, since in the actual structures there is no notion of distance, just relative ordering.) Our freedom to move the tile set  at $(0,0)$ independently of the tile set  at $(1,0)$ stems from the fact that they form disjoint $<_2$-intervals. 

Also note that it is crucial that given an enumeration of the coordinates, the connector intervals in $A$ and the tile sets in $B$ occur in the same $<_2$-order. For if $A$ were unchanged but the $(0,0)$-tile set  appeared $<_2$-above the $(1,0)$-tile set  in $B$, then it would be impossible for both the connector intervals in $A$ to capture tiles from the corresponding tile set. In the canonical models, the disjointness and ordering are ensured by point (4), requiring that the connector intervals and tile set intervals are $<_2$-increasing when ordered antilexicographically. 

\section{Constraints}

For a given string tiling problem $\TT$, we now describe the forbidden substructures in our class $\PP_\TT$, which will impose constraints on arbitrary 3-dimensional permutations in our class and how they may be jointly embedded. We will give some discussion of the constraints before listing them, dividing them into three groups.

The first group of constraints is meant to ensure that when we attempt to perform $<_1$-joint embedding on the canonical models, we must produce a solution to $\TT$. This includes Constraints \ref{c:originintersect}, \ref{c:propagateintersect}, and \ref{c:rules}. In \S \ref{sub:graphs}, we noted that we would wish our constraints to force a grid point to be tiled using a tile from a tile set  with the same coordinates. However, as we are forbidding a \emph{finite} number of finite structures, our constraints must have a \emph{local} character. Since determining the coordinates of a grid point requires walking back to the origin, and thus looking at an unbounded number of vertices, we cannot use our constraints to directly say that a grid point with given coordinates should be tiled using the tile set with the same coordinates. Instead, we will start the tiling at the origin (Constraint \ref{c:originintersect}), and then propagate it by local constraints (Constraint \ref{c:propagateintersect}). Finally, Constraint \ref{c:rules} ensures this tiling gives a valid solution of $\TT$.

The next group of constraints is meant to ensure that elements of $\PP_\TT$ look like the canonical models in certain ways. 
This includes Constraints \ref{c:path}--\ref{c:no01}. These ensure that the origin, path, and grid points/tile sets actually encode something grid-like. They also enforce some conventions we chose for the canonical models, such as that the grid points/tile sets should be antilexicographically increasing in $<_2$. We would like to demand even more from the structures in $\PP_\TT$: for example, we would like every path to have an origin point, or every grid point to have coordinates from a path. However, we cannot enforce such ``totality'' conditions since $\PP_\TT$ must be closed under substructure, so we must allow for partial structures.

In addition to allowing for partial grids, we don't impose any constraints to ensure the uniqueness of grids or paths, so a structure in $\PP_\TT$ can have multiple distinct paths each with its own grid attached. While we could enforce uniqueness by forbidding substructures, this would require us to make points in distinct factors equal when performing joint embedding (for example, if we demand a unique path origin, then if there is a path origin in each factor, they must be identified with each other). While this may be viable, it seemed as though it would greatly increase the complexity of the arguments.

Our final group of constraints  is concerned with the intersection of special intervals, primarily to control the interaction between multiple copies of a grid, some of which may be partial grids. For example, if we have multiple grids of connector intervals in a single structure, all connector intervals with given coordinates $(x, y)$ should intersect, otherwise we could not insert a given tile into all of them when performing joint embedding. This group, which will play a more technical role, 
includes Constraints \ref{c:originintersect}--\ref{c:axisintersect}.

The reason for the overlap between the first and last groups is that forcing tilings is essentially a special case of forcing the intersection of special intervals. If $I_G$ is a connector interval and $I_T$ a tile interval, then the grid point of $I_G$ is tiled by a tile from $I_T$ if and only if $I_G$ and $I_T$ intersect and $I_G <_{1,3} I_T$.



Given a string tiling problem $\TT$, we now define a class $\PP_\TT$ of 3-dimensional permutations by forbidding substructures to enforce the constraints below. Some constraints are initially described informally, with the formal description of the forbidden substructures nested below. For a few other constraints, we describe the forbidden substructures after the entire list.

\begin{enumerate}
\item \label{c:path} Path points have at most 1 predecessor and at most 1 successor.
\item \label{c:originpred} Path origins have no predecessor.
\item \label{c:coords} Special intervals are coordinatized by at most one pair of points.
\item \label{c:pathbelow} Path points, and their associated copies of $E_P^i$ (or $E_O^i$), are $<_2$-below all copies of $E^0_G$ and $E^1_T$.

\item \label{c:speciallex} 
Special intervals are antilexicographically increasing in $<_2$ with respect to their coordinates.
\begin{enumerate}
\item Let $I, I'$ be a pair of special intervals, with $I'$ a horizontal or vertical successor of $I$. Then $I <_2 I'$.
\item Let $I, I'$ be a pair of special intervals. Suppose that $I'$ is on the $y$-axis, and $I$ has a horizontal predecessor $I_{hp}$ with $I_{hp} <_2 I'$. Then $I <_2 I'$.
\end{enumerate}

\item \label{c:no01} No point can belong to a copy of both a $0$-superscripted and $1$-superscripted element of $\EE$

\item \label{c:originintersect} All special intervals corresponding to grid-origins or tile-origins intersect. Furthermore, if $I^0$ corresponds to a grid-origin and $I^1$ to a tile-origin and $I^0 <_1 I^1$, then $I^0 <_3 I^1$.

\item \label{c:propagateintersect} Two special intervals must intersect if their respective predecessors intersect.

Let $I, I'$ be special intervals.
\begin{enumerate}
\item Suppose $I$ is on neither the $x$ nor $y$-axis. Suppose $I$ has horizontal predecessor $I_{hp}$ and vertical predecessor $I_{vp}$, and $I'$ has horizontal predecessor $I'_{hp}$ and vertical predecessor $I'_{vp}$. If $I_{hp}$ intersects $I'_{hp}$ and $I_{vp}$ intersects $I'_{vp}$, then $I$ must intersect $I'$.
\item Suppose $I$ is on the $x$-axis. Suppose $I$ has horizontal predecessor $I_{hp}$ and $I'$ has horizontal predecessor $I'_{hp}$. If $I_{hp}$ intersects $I'_{hp}$, then $I$ must intersect $I'$. 
\item Suppose $I$ is on the $y$-axis. Suppose $I$ has vertical predecessor $I_{vp}$ and $I'$ has vertical predecessor $I'_{vp}$. If $I_{vp}$ intersects $I'_{vp}$, then $I$ must intersect $I'$. 
\end{enumerate}

Furthermore, if $I <_1 I'$ in any of the above cases, then $I <_3 I'$.

\item \label{c:intersectpred} If two special intervals intersect, then their respective predecessors must intersect.

Let $I, I'$ be special intervals (allowing $I=I'$).
\begin{enumerate}
\item Suppose $I$ is on neither the $x$ nor $y$-axis. Suppose $I$ has horizontal predecessor $I_{hp}$ and vertical predecessor $I_{vp}$, and $I'$ has horizontal predecessor $I'_{hp}$ and vertical predecessor $I'_{vp}$. If $I$ intersects $I'$, then $I_{hp}$ must intersect $I'_{hp}$ and $I_{vp}$ must intersect $I'_{vp}$.
\item Suppose $I$ is on the $x$-axis. Suppose $I$ has horizontal predecessor $I_{hp}$ and $I'$ has horizontal predecessor $I'_{hp}$. If $I$ intersects $I'$, then $I_{hp}$ must intersect $I'_{hp}$. 
\item Suppose $I$ is on the $y$-axis. Suppose $I$ has vertical predecessor $I_{vp}$ and $I'$ has vertical predecessor $I'_{vp}$. If $I$ intersects $I'$, then $I_{vp}$ must intersect $I'_{vp}$. 
\end{enumerate}

\item \label{c:specialintersect} If $I_1, I_2$, and $I_3$ are special intervals, and $I_1$ and $I_2$ intersect $I_3$, then $I_1$ and $I_2$ intersect.

\item \label{c:axisintersect} Let $I$ and $I'$ be special intervals that intersect, and suppose $I$ is on the $x$-axis (resp. $y$-axis). If $I'$ is coordinatized, then $I'$ is on the $x$-axis (resp. $y$-axis).

\item \label{c:rules} The tiling rules of $\TT$ are respected.

\end{enumerate}

We now describe the forbidden substructures corresponding to Constraints \ref{c:originpred}, \ref{c:coords}, \ref{c:axisintersect}, and \ref{c:rules}.

For Constraint \ref{c:originpred}, we first forbid any structure consisting of a copy of $E^i_O$ with its root captured by a copy of $E^i_P$. Note that this corresponds to several forbidden substructures, since there are many ways the points of the copy of $E^i_O$ can interleave with those of the copy of $E^i_P$ in $<_2$, while still having the root of $E^i_O$ be captured. This will be true in all other constraints, i.e. the forbidden configurations we describe will correspond to many forbidden substructures, depending on how the points of the copies of the elements of $\EE$ interleave. For Constraint \ref{c:originpred}, we also forbid any structure consisting of a copy of $E^i_O$ with its root captured by another copy of $E^i_O$. 

For Constraint \ref{c:coords}, we forbid any structure consisting of  two copies of $E^i_X$ (or of $E^i_Y$) intersecting at their root. We also forbid any copy of $E^i_X$ or $E^i_Y$ capturing two points that are in $P^i$.

For Constraint \ref{c:axisintersect}, we forbid any structure consisting of the intervals $I$ and $I'$ as described, where the $x$-coordinate (resp. $y$-coordinate) of $I$ is in $O^i$ and the $x$-coordinate (resp. $y$-coordinate) of $I'$ is in $P^i\bs O^i$. 

For Constraint \ref{c:rules}, suppose $\TT$ forbids a tile of type $j$ at distance $d$ to the right of (resp. above) a tile of type $i$. Then we forbid any substructure consisting of $g, g' \in G^0$ such that $g'$ is a $d$-fold horizontal (resp. direct vertical) successor of $g$ (note that this implies the existence of suitable copies of $E^0_G$, $E^0_X, E^0_Y$, and $E^0_P$), $t \in T^1_i, t' \in T^1_j$ such that $t'$ is the $d$-fold horizontal (resp. direct vertical) successor of $t$, and finally $g$ is tiled by $t$ and $g'$ is tiled by $t'$.

\section{Weak coordinates}

When we perform joint embedding on two structures $A$ and $B$, where $A$ contains a $G^0$-grid, and $B$ a grid of tile-sets, Constraints \ref{c:originintersect} and \ref{c:propagateintersect} will force that the connector intervals in the $G^0$-grid in $A$ are tiled using points from $B$, since tiling is a particular case of special intervals intersecting, as discussed before the constraint list. However, connector intervals may be forced to capture tiles for other reasons.

 Consider the following scenario. There is a connector interval $I \subset A$ that is part of the $G^0$-grid and another connector interval $I' \subset A$ that is part of another $G^0$-grid that is missing a grid-origin. Constraints \ref{c:originintersect} and \ref{c:propagateintersect} will not force us to tile $I'$. However, it may be that the endpoints of $I$ are $<_2$-between the endpoint of $I'$, so by tiling $I$ we must also inadvertently tile $I'$. Then, if $I'$ has successors in its own partial $G^0$-grid, Constraint \ref{c:propagateintersect} takes effect and we may be forced to tile them as well. 

We see that in addition to the tiling of a connector interval being forced by the usual propagation along coordinate paths, the tiling can also be forced due to intersection properties, and then propagate as usual. Thus, in addition to considering a special interval to have coordinates $(x, y)$ if it is coordinatized by the $x^{th}$ and $y^{th}$ points on a path with a path-origin, we will also want to consider all special intervals that intersect such intervals to have coordinates $(x, y)$.

\begin{definition}
Given a special interval $I$, we say $I$ is \textit{weakly coordinatized by $(x, y) \in \N^2$} if one of the following cases holds.
\begin{enumerate}
\item $(x, y)  = (0,0)$: $I$ is, or intersects, a grid-origin or tile-origin.
\item $x=0, y\neq 0$: $I$ has a vertical predecessor weakly coordinatized by $(0, y-1)$ or intersects some special interval $I'$ with such a predecessor.
\item $x\neq0, y=0$: $I$ has a horizontal predecessor weakly coordinatized by $(x-1, 0)$ or intersects some special interval $I'$ with such a predecessor.
\item $x,y \neq 0$: $I$ has a horizontal predecessor weakly coordinatized by $(x-1, y)$ and a vertical predecessor weakly coordinatized by $(x, y-1)$, or intersects some special interval $I'$ with such predecessors.
\end{enumerate}

We say a point is \textit{weakly coordinatized by $(x, y)$} if it is an endpoint of some special interval, i.e. either part of a tile set or one of the two $<_2$-least points of $E^0_G$, weakly coordinatized by $(x, y)$.
\end{definition} 

In the definition above, instead of specifying that $I$ has weak coordinates $(x, y)$ if it intersects an appropriate $I'$, we might have only required that there is a chain of special intervals $(I_0 = I, I_1, \dots, I_n = I')$ with each $I_i$ intersecting $I_{i+1}$. But then by Constraint \ref{c:specialintersect}, we would have $I$ intersects $I'$, so this does not yield anything new.



We will now show that several properties enforced by our constraints for our earlier notion of coordinates will also hold for weak coordinates. 

\begin{lemma} \label{lemma:wc unique}
The weak coordinates of special intervals are unique.
\end{lemma}
\begin{proof}
Suppose $I$ has weak coordinates $(x, y)$ and $(x', y')$. First, suppose $(x, y) = (0,0)$. Then $I$ must intersect a grid origin or path origin $J$ (allowing $I = J$), and also $I$ intersects a special interval $J'$ (allowing $I=J'$) such that $J'$ has predecessor(s) with weak coordinates $(x'-1,y')$ and/or $(x',y'-1)$, so $J'$ has coordinates on a path. By Constraint \ref{c:specialintersect}, $J$ and $J'$ intersect. By Constraint \ref{c:axisintersect}, the $x$ and $y$-coordinates of $J'$ must be path origins, and so cannot have predecessors, which is a contradiction.

Now suppose $(x, y), (x', y') \neq (0,0)$, with $(x, y) <_{antilex} (x',y')$. We will further suppose $x,x' \neq 0, y,y' \neq 0$, although we will return to these cases afterward. By induction, we may assume all special intervals with weak coordinates antilexicographically less than $(x, y)$ have unique weak coordinates.

Then we may find special intervals $J$ and $J'$ (possibly equal to $I$) such that the following hold.
\begin{enumerate}[(i)]
\item $J$ and $J'$ intersect $I$, and thus intersect each other.
\item $J$ has a horizontal predecessor weakly coordinatized by $(x-1, y)$ and a vertical predecessor weakly coordinatized by $(x, y-1)$.
\item $J'$ has a horizontal predecessor weakly coordinatized by $(x'-1, y')$ and a vertical predecessor weakly coordinatized by $(x', y'-1)$.
\end{enumerate}

 As $J$ and $J'$ intersect, by Constraint \ref{c:intersectpred} the horizontal predecessor of $J$ must intersect that of $J'$, and similarly for vertical predecessors. By induction, we may assume the predecessors of $J$ and $J'$ have unique weak coordinates. Thus we have $x=x', y=y'$. 
 
 In the case $y=0$ (the case $x=0$ is similar), we must also have that $y'=0$ by Constraint \ref{c:axisintersect}. We then only get horizontal predecessors for $J$ and $J'$, but we may still finish as in the previous case.
\end{proof}

\begin{lemma} \label{lemma:wc inter}
All special intervals weakly coordinatized by $(x, y)$ intersect.
\end{lemma}
\begin{proof}
We proceed by antilexicographic induction on $(x, y)$. If $(x, y) = (0,0)$, then this is immediate from Constraint \ref{c:originintersect}.
 
 Otherwise, assume $x,y \neq 0$ (as in Lemma \ref{lemma:wc unique}, these cases just require using Constraint \ref{c:axisintersect} and a single predecessor), and let $I_1, I_2$ have weak coordinates $(x, y)$. Then $I_1$ intersects a special interval $I'_1$ such that $I'_1$ has a horizontal predecessor weakly coordinatized by $(x-1, y)$ and a vertical predecessor weakly coordinatized by $(x, y-1)$, and $I_2$ similarly intersects some interval $I'_2$. By induction, the respective predecessors intersect. Thus by Constraint \ref{c:propagateintersect}, $I_1'$ and $I_2'$ intersect, and so $I_1$ and $I_2$ intersect by Constraint \ref{c:specialintersect}.
\end{proof}

\begin{corollary} \label{cor:wc}
(1) Suppose $a$ is the endpoint of a special interval $I$ and is $<_2$-between 2 points weakly coordinatized by $(x, y)$. Then $a$ is weakly coordinatized by $(x, y)$.

(2) All 1-tiles weakly coordinatized by $(x, y)$ are $<_2$ all 2-tiles weakly coordinatized by $(x, y)$.

(3) Suppose $I$ is weakly coordinatized by $(x, y)$. If $x,y \neq 0$ and $I_{hp}$ and $I_{vp}$ are horizontal and vertical predecessors of $I$, then $I_{hp}$ is weakly coordinatized by $(x-1, y)$ and $I_{vp}$ by $(x, y-1)$. If $y=0$ (resp. $x = 0$), the same holds, but only with $I_{hp}$ (resp. $I_{vp}$).
\end{corollary}
\begin{proof}
$(1)$ Suppose $a$ is $<_2$-between $b,c$ weakly coordinatized by $(x, y)$. If $b, c$ belong to the same copy of an element of $\EE$, then that copy intersects $I$, and we are done by Lemma \ref{lemma:wc inter}. If $b,c$ belong to different copies of elements of $\EE$, their respective special intervals intersect each other by Lemma \ref{lemma:wc inter}, and so intersect $I$, and we are again done.

$(2)$ If not, there would be a pair of non-intersecting tile-intervals weakly coordinatized by $(x, y)$, contradicting Lemma \ref{lemma:wc inter}.

$(3)$ If not, the weak coordinates of $I$ would not be unique, contradicting Lemma \ref{lemma:wc unique}.
\end{proof}

\begin{lemma} \label{lemma:wc lex}
Suppose $I$ is weakly coordinatized by $(x, y)$, $I'$ is weakly coordinatized by $(x', y')$, and $(x, y) <_{antilex} (x', y')$. Then $I <_2 I'$.
\end{lemma}
\begin{proof}
 Fix $I$ with weak coordinates $(x, y)$ and $I'$ with weak coordinates $(x', y')$. By induction, it is sufficient to consider the cases $(x',y') = (x+1, y)$ and $(x',y')=(x',y+1)$.
 
\begin{claim1}
Let $I, I', J, J'$ be special intervals. Suppose $I$ intersects $I'$, $J$ intersects $J'$, and $I' <_2 J'$. Then $I <_2 J$.
\end{claim1}
\medskip
{\noindent\textbf{Proof of Claim:} \ignorespaces} Suppose not. Then $I$ must intersect $J$. But then by Constraint \ref{c:specialintersect}, $I'$ must intersect $J'$. 
\hfill $\lozenge$
\medskip

First assume $(x',y') = (x+1, y)$. Then $I'$ intersects some interval $J'$ with a horizontal predecessor $J'_{hp}$ weakly coordinatized by $(x, y)$, which in turn intersects $I$ by Lemma \ref{lemma:wc inter}. By Constraint \ref{c:speciallex}(a), we have $J'_{hp} <_2 J'$, and so $I <_2 I'$ by the Claim. By induction, we get the same result for $(x',y') = (x+i, y), i>0$.

The case $(x',y')=(x',y+1)$ is similar, though more involved.  Then, $I'$ intersects some interval $J'$ with a vertical predecessor $J'_{vp}$ weakly coordinatized by $(x', y)$. If $x < x'$, then by the previous case, $I <_2 J'_{vp}$, and if $x = x'$ then $I$ intersects $J'_{vp}$ by Lemma \ref{lemma:wc inter}. As $J'_{vp} <_2 J'$ by Constraint \ref{c:speciallex}(a), the Claim gives $J'_{vp} <_2 I'$, and so $I <_2 I'$.

So suppose $x = x'+i$, $i \geq 0$. It suffices to consider the case $x' = 0$, since increasing $x'$ only increases the $<_2$-position of $I'$, by the first case. We proceed by induction on $i$, with the case $i=0$ handled above. We get $J'_{vp}$ as above, and similarly get that $I$ intersects some interval $J$ with a horizontal predecessor $J_{hp}$ weakly coordinatized by $(x'+(i-1),y)$. By induction, $J_{hp} <_2 I'$. Then by Constraint \ref{c:speciallex}(b), $J <_2 I'$, so the Claim gives $I <_2 I'$. 
\end{proof}

\section{Reductions}
\subsection{Reductions with the $<_1$-JEP} \label{sub:<1}
We begin with the easy direction, that if $\PP_\TT$ has the $<_1$-JEP then performing joint embedding on our canonical models will encode a solution to $\TT$.

\begin{lemma} \label{lemma:1JEPtotile}
	Let $\TT$ be a string tiling problem, and $\PP_\TT$ the corresponding 3-dimensional permutation class. If $\PP_\TT$ has the $<_1$-JEP, then $\TT$ has a solution. 
\end{lemma}
\begin{proof}
	Let $A^*_{<_1}$ and $B^*_{<_1}$ be the canonical models from \S \ref{sub:canon}. Then $A^*_{<_1}, B^*_{<_1} \in \PP_\TT$, so we can apply the $<_1$-JEP yielding $C^*_{<_1}$. As $A^*_{<_1} <_1 B^*_{<_1}$ there can be no identifications of points between the factors, so we may assume $C^*_{<_1}$ has $A^*_{<_1} \sqcup B^*_{<_1}$ as a base set. Furthermore, by Constraint \ref{c:originintersect}, the grid-origin in $A^*_{<_1}$ must capture some tile from the tile-origin in $B^*_{<_1}$. This then propagates to a tiling of the entire grid in $A^*_{<_1}$ by tiles from the grid in $B^*_{<_1}$ by Constraint \ref{c:propagateintersect}, while respecting the rules of the tiling problem by Constraint \ref{c:rules}. We thus associate to $C^*_{<_1}$ the tiling $\theta(x, y) = i$ if the connector interval associated to the $G^0$-point with coordinates $(x, y)$ captures a tile of type $i$ (if it captures tiles of both types, we may pick either).
\end{proof}

Now we must assume $\TT$ has a solution and show $\PP_\TT$ has the $<_1$-JEP. The difficulties in defining a joint embedding procedure arise from the way in which arbitrary $A, B \in \PP_\TT$ can deviate from the canonical models. We have already mentioned that an element of $\PP_\TT$ can have multiple grids and partial grids, which led us to consider weak coordinates. Another possible issue that arises from this is that a structure can contain an incorrect tiling. For a simple example, consider a tiling problem where the only constraints are that a 2-tile cannot be adjacent to either a 1-tile or a 2-tile, so the unique solution is to tile the whole grid with 1-tiles. However, $A \in \PP_\TT$ can contain a full grid of connector intervals, and a single tile-set at the origin, such that the connector interval at the origin captures the 2-tile. Then if $B \in \PP_\TT$ contains a full grid of tile sets, when we joint-embed $A$ and $B$, we will have to arrange so that the connector interval at the origin of $A$ captures the 1-tile from the grid of tile-sets in $B$, while still capturing the 2-tile from the partial grid of tile sets in $A$, and making the tile origins in $A$ and $B$ intersect. This will be possible, but relies on having only two types of tile, which is why we imposed this restriction.

From Lemma \ref{lemma:wc lex}, we see that the special intervals weakly coordinatized by $(x,y)$ are contained in a particular $<_2$-interval, separated from all other special intervals with distinct weak coordinates, and that the intervals are antilexicographically increasing in $<_2$. Thus an arbitrary element of $\PP_\TT$ recovers some of the structure of Figure \ref{fig:canonical3}, although with possibly many special intervals at a given weak coordinate. As is already done Figure \ref{fig:canonical3}, our first step when performing joint embedding will be to put the intervals in the two factors with the same weak coordinates at roughly the same $<_2$-level.  This is done in the following definition, where we would like to simply say that the bottom coordinates of the $<_2$-intervals $I_A$ and $I_B$ are set equal with respect to $<_2$, as are the top coordinates. However, as distinct points cannot be equal with respect to $<_2$, the definition is more convoluted.

\begin{definition}
Let $C$ be a structure equipped with a partial order $<$, and let $A, B \subset C$ be totally $<$-ordered. Let $I_A, I_B$ be closed $<$-intervals in $A,B$. Extending $<$ such that $b_1 < I_A < b_2$ for any $b_1 < I_B < b_2$, and such that $a_1 < I_B < a_2$ for any $a_1 < I_A < a_2$, will be called \textit{$<$-aligning $I_A$ with $I_B$}. Note, this may not be possible, depending on the initial $<$-configuration.
\end{definition}

Given $A, B$, we will use the definition in our joint embedding procedure as follows. After taking the disjoint union $C = A \sqcup B$, $<_2$ will be a partial order on $C$. We will partition $A$ into $<_2$-intervals $I_{A, i}$ for $i \in \N$, with the condition that if $i < j$ then the $I_{A, i} <_2 I_{A, j}$, and similarly partition $B$ into $<_2$-intervals $I_{B, i}$. For each $i$, we will then align $I_{A, i}$ with $I_{B, i}$. This yields a sequence of disjoint increasing $<_2$-intervals in $C$, and we will then complete $<_2$ to a linear order on $C$ by completing it on each such interval separately. By the disjointness, the completion in any one interval can be done independently of the completion on other intervals.

Before beginning our next lemma, we repeat that the claim at the beginning of its proof is the reason we use a third linear order in this paper.

\begin{lemma} \label{lemma:tilingtoJEP}
Let $\TT$ be a string tiling problem, and $\PP_\TT$ the corresponding 3-dimensional permutation class. If $\TT$ has a solution, then $\PP_\TT$ has the $<_1$-JEP.
\end{lemma}
\begin{proof}
Let $A, B \in \PP_\TT$, and begin by defining $C$ to be the disjoint union of $A$ and $B$, so the order relations are only partially defined. Next, complete  $<_1$ and $<_3$ in $C$ so that $A <_1 B$ and $A <_3 B$. Thus, all that remains is to determine how points in $A$ are $<_2$-related to points in $B$, while satisfying the constraints.

\begin{claim}
For any $<_2$-completion of $C$, and any $E \subset C$ a copy of some element of $\EE$, either $E \subset A$ or $E \subset B$.
\end{claim}
\medskip
{\noindent\textbf{Proof of Claim:} \ignorespaces} Suppose $a \in E$ and $a \in A$. As $A <_1 B$, all points $a' \in E$ with $a' <_1 a$ are in $A$. As ${<_1} = {<_3^{opp}}$ on $E$, for any $a' \in E$ such that $a' >_1 a$ we have $a' <_3 a$; as $B >_3 A$, such $a'$ are also in $A$.
\hfill $\lozenge$
\medskip

We now extend $<_2$ in $C$ so that all copies of $E^i_O$ and $E^i_P$ in $B$ are $<_2$-below all points in $A$, and similarly so that all copies of $E^i_O$ and $E^i_P$ in $A$ are $<_2$-below all copies of $E^0_G$ and $E^1_T$ in $B$ (here we use Constraint \ref{c:pathbelow}). In Figure \ref{fig:canonical3}, this is satisfied if we erase the dividing line. Given this extension, we prove the next claim.

\begin{claim} \label{cl:no coord}
For any $<_2$-completion of $C$, let $E \subset C$ be a copy of some element of $\EE$ in one factor. Then $E$ captures no $P^i$-points in the other factor.
\end{claim}
\medskip
{\noindent\textbf{Proof of Claim:} \ignorespaces} 
If $E \subset B$, then it captures no points in $A$, as $A <_1 B$. If $E \subset A$, it captures no $P^i$-points in $B$, as all such points are $<_2$-below all points in $A$.
\hfill $\lozenge$
\medskip

Constraints \ref{c:path}--\ref{c:coords} and \ref{c:speciallex}(a) follow immediately from the claims above and the fact that the constraints hold in each factor. Constraint \ref{c:pathbelow} holds by the paragraph before Claim \ref{cl:no coord}, and Constraint \ref{c:no01} holds as we have identified no points.

The remaining constraints concern the relations between special intervals. For each $(x, y) \in \N^2$, we may consider the closed $<_2$-interval $I^A_{x, y}$, whose endpoints are the $<_2$-least and greatest points weakly coordinatized by $(x, y)$ in $A$, and similarly $I^B_{x, y}$. By Lemma \ref{lemma:wc lex}, in each factor these intervals are non-overlapping and antilexicographically increasing with respect to $<_2$. 
We may thus $<_2$-align each $I^A_{x, y}$ with $I^B_{x, y}$, and set $I^X_{x,y} <_2 I^Y_{x',y'}$ for $X, Y \in \set{A, B}$ and $(x,y) <_{antilex} (x',y')$ (as in Figure \ref{fig:canonical3}). From this, it follows that Constraint \ref{c:speciallex}(b) is satisfied. Thus the Constraints \ref{c:path}--\ref{c:no01} that we have discussed so far will remain satisfied for any completion of $<_2$. 

We now consider each coordinate-pair $(x, y)$ one at a time, and determine the $<_2$-order of the points weakly coordinatized by $(x, y)$ independently of what we do at other weak coordinates.
 
Let $\theta \colon \N^2 \to \set{1,2}$ be a valid tiling. For now, we assume there is a connector interval in $A$ and tile set in $B$, each weakly coordinatized by $(x, y)$.

Suppose $\theta(x, y) = 1$. We will work entirely in $I^A_{x, y}$ and $I^B_{x, y}$ (and by Corollary \ref{cor:wc}(1), all special interval endpoints in these intervals are weakly coordinatized by $(x, y)$). Figure \ref{fig:joint} shows an example of the joint embedding procedure at a coordinate $(x, y)$ with $\theta(x, y) = 1$, where we have zoomed in on the special intervals. In Figure \ref{fig:joint}, in $A$ there is a connector interval capturing a 2-tile from a tile set and in $B$ there is a connector interval capturing a 1-tile. In $C$, the connector interval from $A$ captures the correct tile from $B$ and all the special intervals intersect.

We now describe the general procedure when $\theta(x,y) = 1$. Let $I_A$ be the intersection of all special intervals in $I^A_{x,y}$, and $I_B$ for $I^B_{x,y}$ (these are non-empty by Lemma \ref{lemma:wc inter}, and are shown in Figure \ref{fig:joint}). We first set all points from $A$ $<_2$-below all the 2-tiles from $B$. Note the bottom endpoint of $I_B$ must be $<_2$-below all the 2-tiles in $B$, as must all the 1-tiles in $B$ by Corollary \ref{cor:wc}(2). Thus we may set $I_A$ to contain all the 1-tiles from $B$ as well as the bottom endpoint of $I_B$. This makes all the connector intervals in $I^A_{x,y}$ capture all 1-tiles in $I^B_{x,y}$, and makes $I_A$ intersect $I_B$ (this intersection is shown in Figure \ref{fig:joint}), so all special intervals in $I^A_{x,y}$ intersect those in $I^B_{x,y}$. Finally, we  arbitrarily complete $<_2$ to a linear order, when restricted to points in $I^A_{x,y} \cup I^B_{x,y}$.

The case $\theta(x, y) = 2$ is similar. 

\begin{figure} 
	\begin{center}
\includegraphics[scale=.4]{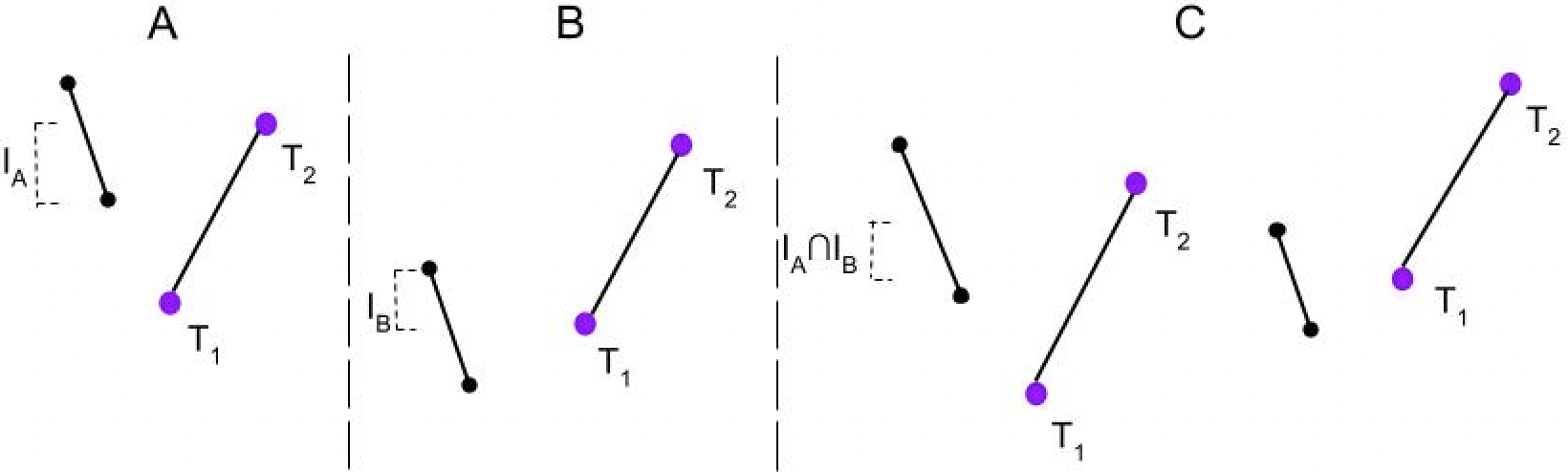}
\end{center}
\caption{An example of joint embedding at $(x, y)$ with $\theta(x, y) = 1$, projected onto $<_1, <_2$. The lines with black endpoints represents connector intervals, while the purple points represent tiles.}
\label{fig:joint}
\end{figure}

If there is no connector interval in $A$ and tile set in $B$, each weakly coordinatized by $(x, y)$, the process is simpler. We just intersect $I_A$ with $I_B$ to ensure all the special intervals in $A$ weakly coordinatized by $(x, y)$ intersect all those in $B$ weakly coordinatized by $(x, y)$.

Because we have made every special interval from $I^A_{x, y}$ intersect every special interval $I^B_{x, y}$, we will satisfy Constraints 
\ref{c:originintersect}--\ref{c:specialintersect}, with Constraint \ref{c:intersectpred} additionally using Corollary \ref{cor:wc}(3). We have also tiled every  $G^0$-point in $A$ weakly coordinatized by $(x, y)$ according to $\theta(x, y)$, and not tiled any $G^0$-point in $B$, and so will satisfy Constraint \ref{c:rules}. Because we have only intersected special intervals on a given axis with those on the same axis, Constraint \ref{c:axisintersect} holds as well.

To finish, we must complete $<_2$ to a linear order while still satisfying Constraints \ref{c:originintersect}--\ref{c:rules}. By inspection, they will still be satisfied if we do not intersect any further special intervals during the completion. Given a special interval $I$ in one factor and $I'$ in the other that do not yet intersect, we will be forced to intersect $I$ and $I'$ only if either there is another special interval $J$ intersecting both $I$ and $I'$, or there are special intervals $J$ and $J'$ both intersecting $I'$ with $I$ $<_2$-between them (or the same configuration with the roles of $I$ and $I'$ reversed). Recall we only intersected special intervals with the same weak coordinates. Since we also intersected all special intervals at a given set of weak coordinates, the first configuration cannot appear.  Since our partially-defined structure satisfies Constraint \ref{c:specialintersect}, $J$ and $J'$ must intersect each other in the second configuration, and thus $I$ cannot be $<_2$-between them.

Thus there is some way to complete $<_2$ to a linear order while not intersecting any further special intervals, and any such completion will suffice to finish defining $C$.
\end{proof}

\subsection{From the $<_1$-JEP to the JEP} \label{sub:jep}

We first describe why we initially restricted ourselves to the $<_1$-JEP. Note that our definition of capture and the final parts of Constraint \ref{c:originintersect} and \ref{c:propagateintersect} are asymmetric with respect to $<_1$. When jointly embedding our canonical models $A^*_{<_1}$ and $B^*_{<_1}$, if we were not forced to put $A^*_{<_1} <_1 B^*_{<_1}$, we could trivially jointly embed them by putting $A^*_{<_1} >_1 B^*_{<_1}$. But then no connector intervals in $A^*_{<_1}$ would capture any tiles in $B^*_{<_1}$, and so this would not encode a solution to the tiling problem. One could try to remove these asymmetries, but then the JEP will fail.



In order to remove the requirement of $<_1$-JEP from Lemma \ref{lemma:1JEPtotile}, we slightly adjust the class $\PP_\TT$ we are working in. For each 0-superscripted element of $\EE$, we introduce a corresponding 2-superscripted element to $\EE$ from $\AA$, and for each 1-superscripted  element of $\EE$ we introduce a corresponding 3-superscripted element to $\EE$ from $\AA$, thus doubling the size of $\EE$. We define the corresponding unary predicates as before.

The idea is that 2-superscripted elements should behave like 0-superscripted ones, and 3-superscripted elements like 1-superscripted ones, with the exception that 0-superscripted grids should be tiled by 1-superscripted tiles while 2-superscripted grids should be tiled by 3-superscripted tiles. We will also use $<_2$ to separate the $0,1$-superscripted elements from $2,3$-superscripted elements.

Thus, given a string tiling problem $\TT$, we define a 3-dimensional permutation class $\QQ_\TT$ as follows. We use all the constraints from $\PP_\TT$, and then duplicate those constraints replacing 0-superscripted and 1-superscripted predicates with 2-superscripted and 3-superscripted predicates, respectively.

 We also add the following constraints.
 \begin{enumerate}
 \item[$\begin{NoHyper}(\ref{c:no01}\end{NoHyper}^*)$] \label{c:no0123} Constraint \ref{c:no01} is replaced by a constraint forbidding the identification of any points from 2 distinctly-superscripted elements of $\EE$.  
 \item[(13)]  All copies of $\set{0,1}$-superscripted elements of $\EE$ must be $<_2$-below all copies of $\set{2,3}$-superscripted elements of $\EE$.
 \end{enumerate}

\begin{lemma} \label{lemma:tiletojep}
Let $\TT$ be a string tiling problem, and $\QQ_\TT$ the corresponding 3-dimensional permutation class. If $\TT$ has a solution, then $\QQ_\TT$ has the JEP.
\end{lemma}
\begin{proof}
Fix a tiling $\theta \colon \N^2 \to \set{1,2}$. Given $A,B$ in our new class, split both into 2 $<_2$-intervals so that the lesser interval contains all copies of $\set{0,1}$-superscripted elements of $\EE$, and the greater interval contains all copies of $\set{2,3}$-superscripted elements of $\EE$. We may then apply the joint embedding procedure of Lemma \ref{lemma:tilingtoJEP} separately to the pair of $<_2$-lesser intervals and the pair of $<_2$-greater intervals.
\end{proof}

In the following lemma, we weaken the $<_1$-JEP from earlier to simply the JEP. This is done by adjusting the canonical models so that we must perform the $<_1$-JEP with either a copy of our earlier canonical models, or with a copy of the earlier canonical models using $\set{2,3}$-superscripted elements instead of $\set{0,1}$-superscripted elements.

\begin{figure}[h]
	\begin{center}
\includegraphics[scale=.4]{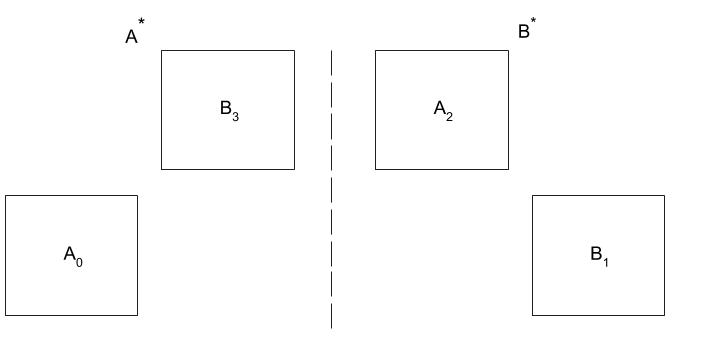}
\end{center}
\caption{The canonical models in $\QQ_\TT$, projected onto $<_1, <_2$.}
\label{fig:double}
\end{figure}

\begin{lemma} \label{lemma:jeptotile}
Let $\TT$ be a string tiling problem, and $\QQ_\TT$ the corresponding 3-dimensional permutation class. If $\QQ_\TT$ has the JEP, then $\TT$ has a solution.
\end{lemma}
\begin{proof}
We describe our new canonical models, which are pictured in Figure \ref{fig:double}. Let $A_0$ be as $A^*_{<_1}$ in Lemma \ref{lemma:1JEPtotile} and $B_3$ be as $B^*_{<_1}$ in Lemma \ref{lemma:1JEPtotile} but with 3-superscripted elements of $\EE$ instead of 1-superscripted elements of $\EE$. Let $A^* = A_0 \sqcup B_3$, with $A_0 <_{1,2,3} B_3$.

Let $A_2$ be as $A^*_{<_1}$ in Lemma \ref{lemma:1JEPtotile} but with 2-superscripted elements of $\EE$ instead of 0-superscripted elements of $\EE$ and $B_1$ be as $B^*_{<_1}$ in Lemma \ref{lemma:1JEPtotile}. Let $B^* = A_2 \sqcup B_1$, with $A_2 <_{1,3} B_1$ and $B_1 <_2 A_2$.

In $A^*$, as $<_1$ and $<_3$ agree between $A_0$ and $B_3$, any copy of $E \in \EE$ that occurs must be contained either in $A_0$ or in $B_3$. Similarly in $B^*$, any copy of $E \in \EE$ that occurs must be contained either in $A_2$ or in $B_1$.

As $A^*_{<_1}, B^*_{<_1}$ in Lemma \ref{lemma:1JEPtotile} were in $\PP_\TT$, $A^*$, $B^*$ will be in $\QQ_\TT$. If $\QQ_\TT$ has the JEP, there is some $C^*$ embedding $A^*, B^*$.

By Constraint \ref{c:no01}$^*$, $C^*$ must contain $A^* \sqcup B^*$. Suppose in $C^*$ that $A_0 <_1 B_1$. Then as in Lemma \ref{lemma:1JEPtotile}, we must produce a tiling. If we don't have $A_0 <_1 B_1$ in $C^*$, then it must be that $A_2 <_1 B_3$, and again we must produce a tiling as in Lemma \ref{lemma:1JEPtotile}.
\end{proof}

\begin{theorem}
	       There is no algorithm that, given a finite set of forbidden 3-dimensional permutations, decides whether the corresponding 3-dimensional permutation class has the JEP.
\end{theorem}
\begin{proof}
	Given a string tiling problem $\TT$, consider the 3-dimensional permutation class $\QQ_\TT$. By Lemmas \ref{lemma:tiletojep} and \ref{lemma:jeptotile}, $\TT$ has a solution if and only if $\QQ_\TT$ has the JEP. As the string tiling problem is undecidable, we are finished. 
	\end{proof}

\begin{corollary}
The JEP is undecidable for $n$-dimensional permutation classes with finitely many forbidden $n$-dimensional permutations, for $n \geq 3$
\end{corollary}
\begin{proof}
We have already shown this for $n=3$, so fix $n > 3$. To any 3-dimensional pattern class $\CC$, we can associate an $n$-dimensional permutation class $L(\CC)$ whose constraints are all expansions of the constraints from $\CC$ to $n$ orders. Also, given any $n$-dimensional permutation, we may consider its reduct to the first 3 orders. Both of these operations preserve that there are only finitely many forbidden $n$-dimensional permutations.

We claim that $\CC$ has the JEP if and only if $L(\CC)$ has the JEP. Suppose $L(\CC)$ has the JEP. Given $A, B \in \CC$, we may expand them to structures in $L(\CC)$, jointly embed the expansions, and then take the reduct, giving a joint embedding of $A, B$. Now suppose $\CC$ has the JEP. Given $A, B \in L(\CC)$ we may jointly embed their reducts, and any expansion of the result will give a joint embedding of $A$ and $B$.
\end{proof}

\section{Concluding Remarks}

We finish by discussing the obstructions to adapting this proof to permutation classes. As mentioned before, the main issue is the loss of an easy proof for the claim at the beginning of Lemma \ref{lemma:tilingtoJEP}. If simply taking the projection of our 3-dimensional joint embedding procedure to the first 2 orders, transitivity will force us to produce many configurations we do not intend to. For example, consider the following situation. Let $F$ be a forbidden permutation, and suppose $F = F_1 \sqcup F_2$ with $F_1 <_{1,2} F_2$ (more elaborate constructions can remove this requirement). Suppose we are performing the $<_1$-JEP on $A, B$, and there are $a \in A$ and $b \in B$ such that we must set $a <_2 b$. We may instead consider the structure $A'$ formed from $A$ by placing $F_1 <_2 a$ and $B'$ formed from $B$ by placing $b <_2 F_2$. If $A'$ and $B'$ are still in our permutation class, then when jointly embedding them, transitivity will force $F_1 <_{1,2} F_2$, and so we will create a copy of $F$.

A perhaps more basic manifestation of the difficulty in ruling out unintended configurations in permutation classes is the question of how to represent arbitrarily large grids, and thus the canonical models, in a class $\PP_\TT$ such that $\PP_\TT$ is not the class of all permutations.

\subsection{Acknowledgments} I thank Gregory Cherlin for many discussions on the material in this paper, and the referee for several corrections and suggestions for improving the presentation.

\nocite{*}
\bibliographystyle{abbrvnat}
\bibliography{jep-bib}

\end{document}